\documentclass{amsart}

\usepackage{amssymb, mathtools, enumerate, todonotes, cancel, tikz, float}
\usepackage[hidelinks]{hyperref}
\usepackage{cleveref}

\DeclareMathOperator{\A}{A}
\DeclareMathOperator{\D}{D}

\DeclareMathOperator{\R}{R}
\DeclareMathOperator{\At}{At}
\DeclareMathOperator{\id}{1'}
\DeclareMathOperator{\down}{\downarrow}

\newcommand{\compo}{\mathbin{;}}
\newcommand{\compl}[1]{\overline{#1}}
\newcommand{\bmeet}{\wedge}
\newcommand{\bjoin}{\vee}
\newcommand{\meet}{\bigwedge}
\newcommand{\join}{\bigvee}
\newcommand{\stcomp}[1]{{#1}^{\mathsf{c}}}
\newcommand{\defn}[1]{\textbf{#1}}
\newcommand{\algebra}[1]{\mathfrak{#1}}

\newcommand{\powerset}{\raisebox{0.45ex}{$\wp$}}

\theoremstyle{plain}
\newtheorem{theorem}{Theorem}[section]
\newtheorem{proposition}[theorem]{Proposition}
\newtheorem{lemma}[theorem]{Lemma}
\newtheorem{corollary}[theorem]{Corollary}

\theoremstyle{definition}
\newtheorem{definition}[theorem]{Definition}
\newtheorem{example}[theorem]{Example}
\newtheorem{remark}[theorem]{Remark}

\title[Complete Representation by Partial Functions]{Complete Representation by Partial Functions for Composition, Intersection and Antidomain}
\author{Brett McLean}
\address{Department of Computer Science, University College London, Gower Street, London WC1E 6BT}
\email{b.mclean@cs.ucl.ac.uk}
\thanks{The author would like to thank his PhD supervisor Robin Hirsch for many helpful discussions.}

\begin{document}

\begin{abstract}
For representation by partial functions in the signature with intersection, composition and antidomain, we show that a representation is meet complete if and only if it is join complete. We show that a representation is complete if and only if it is atomic, but that not all atomic representable algebras are completely representable. We show that the class of completely representable algebras is not axiomatisable by any existential-universal-existential first-order theory. By giving an explicit representation, we show that the completely representable algebras form a basic elementary class, axiomatisable by a universal-existential-universal sentence.

\end{abstract}

\maketitle

\section{Introduction}

Whenever we have a concrete class of algebras whose operations are set-\linebreak theoretically defined, we have a notion of a representation: an isomorphism from an abstract algebra to a concrete algebra. Then the representation class\textemdash the class of representable algebras\textemdash becomes an object of interest itself.

One possibility is for the concrete algebras to be algebras of partial functions, and for this scenario various signatures have been considered. Often, the representation classes have turned out to be finitely-axiomatisable varieties or quasi-varieties \cite{schein, 1018.20057, 1182.20058, DBLP:journals/ijac/JacksonS11}.

Extra conditions we can impose on a representation are to require that it be meet complete or to require that it be join complete. A representation is meet complete if it turns any existing infima into intersections and join complete if it turns any existing suprema into unions. Hence we can define meet-complete representation classes and join-complete representation classes. In many important cases these two classes coincide. Bounded distributive lattices represented as rings of sets is an example where they do not \cite{egrot}.

In \cite{journals/jsyml/HirschH97a}, Hirsch and Hodkinson showed that when the representation class is elementary, the complete representation class may (as is the case for Boolean algebras represented as fields of sets) or may not (relation algebras by binary relations) also be elementary.

In this paper we investigate complete representation by partial functions for the signature $(\compo, \bmeet, \A)$ of composition, intersection and antidomain. In \Cref{rep} we see that for this particular signature the algebras behave in many ways like Boolean algebras. We show that, as one consequence of this similarity to Boolean algebras, a representation by partial functions is meet complete if and only if it is join complete.

In \Cref{atomicity} we show that a representation is complete if and only if it is atomic. We use the requirement that completely representable algebras be atomic to prove that the class of completely representable algebras is not closed under subalgebras, directed unions or homomorphic images and is not axiomatisable by any existential-universal-existential first-order theory.

In \Cref{dist} we investigate the validity of various distributive laws with respect to the classes of representable and completely representable $(\compo, \bmeet, \A)$-algebras. This enables us to give an example of an algebra that is representable and atomic, but not completely representable.

In \Cref{arep} we present an explicit representation, which we use, in \Cref{axiom}, to prove our main result: the class of completely representable algebras is a basic elementary class, axiomatisable by a universal-existential-universal first-order sentence.

\section{Representations and Complete Representations}\label{rep}

In this section we give preliminary definitions and then proceed to show that for the signature $(\compo, \bmeet, \A)$, a representation by partial functions is meet complete if and only if it is join complete.

Given an algebra $\algebra{A}$, when we write $a \in \algebra{A}$ or say that $a$ is an element of $\algebra{A}$, we mean that $a$ is an element of the domain of $\algebra{A}$. Similarly for the notation $S \subseteq \algebra{A}$ or saying that $S$ is a subset of $\algebra{A}$. We follow the convention that algebras are always nonempty. If $S$ is a subset of the domain of a map $\theta$ then $\theta[S]$ denotes the set $\{\theta(s) \mid s \in S\}$. If $S_1$ and $S_2$ are subsets of the domain of a binary operation $*$ then $S_1 * S_2$ denotes the set $\{s_1 * s_2 \mid s_1 \in S_1 \text{ and } s_2 \in S_2\}$.

\begin{definition}\label{first}
Let $\sigma$ be an algebraic signature whose symbols are a subset of $\{\compo, \bmeet, 0, \id, \D, \R, \A\}$. An \defn{algebra of partial functions} of the signature $\sigma$ is an algebra of the signature $\sigma$ whose elements are partial functions and with operations given by the set-theoretic operations on those partial functions described in the following.

Let $X$ be the union of the domains and ranges of all the partial functions. We call $X$ the \defn{base}. In an algebra of partial functions
\begin{itemize}
\item
the binary operation $\compo$ is composition of partial functions:
\[f \compo g = \{(x,z) \in X^2 \mid \exists y \in X : (x, y) \in f\text{ and }(y, z) \in g\}\text{,}\]
\item
the binary operation $\wedge$ is intersection:
\[f \bmeet g = \{(x,y) \in X^2 \mid (x, y) \in f\text{ and }(x, y) \in g\}\text{,}\]
\item
the constant $0$ is the nowhere-defined function:
\[0 = \emptyset\text{,}\]
\item
the constant $\id$ is the identity function on $X$:
\[\id = \{(x, x) \in X^2 \}\text{,}\]
\item
the unary operation $\D$ is the operation of taking the diagonal of the domain of a function:
\[\D(f) = \{(x, x) \in X^2 \mid \exists y \in X : (x, y) \in f\}\text{,}\]
\item
the unary operation $\R$ is the operation of taking the diagonal of the range of a function:
\[\R(f) = \{(y, y) \in X^2 \mid \exists x \in X : (x, y) \in f\}\text{,}\]
\item
the unary operation $\A$ is the operation of taking the diagonal of the antidomain of a function\textemdash those points of $X$ where the function is not defined:
\[\A(f) = \{(x, x) \in X^2 \mid \cancel{\exists} y \in X : (x, y) \in f\}\text{.}\]
\end{itemize}
\end{definition}

The list of operations in \Cref{first} does not exhaust those that have been considered for partial functions, but does include the most commonly appearing operations.

\begin{definition}
Let $\algebra{A}$ be an algebra of one of the signatures specified by \Cref{first}. A \defn{representation of $\algebra{A}$ by partial functions} is an isomorphism from $\algebra{A}$ to an algebra of partial functions of the same signature. If $\algebra{A}$ has a representation then we say it is \defn{representable}.
\end{definition}

\begin{theorem}[Jackson and Stokes \cite{DBLP:journals/ijac/JacksonS11}]\label{thm:jackson-stokes}
The class of $(\compo, \bmeet, \A)$-algebras representable by partial functions is a finitely-based variety.
\end{theorem}

In fact in \cite{DBLP:journals/ijac/JacksonS11} a finite equational axiomatisation of the representation class is given, implicitly. So there exist known examples of such axiomatisations.

If an algebra of the signature $(\compo, \bmeet, \A)$ is representable by partial functions, then it forms a $\bmeet$-semilattice. Whenever we treat such an algebra as a poset, we are using the order induced by this semilattice.

The next two definitions apply to any situation where the concept of a representation has been defined. So in particular, these definitions apply to representations as fields of sets as well as to representations by partial functions.

\begin{definition}\label{def:meet}
A representation $\theta$ of a poset $\algebra{P}$ over the base $X$ is \defn{meet complete} if, for every nonempty subset $S$ of $\algebra{P}$, if $\meet S$ exists, then
\[\theta(\meet S) = \bigcap \theta[S]\text{.}\]
\end{definition}

\begin{definition}\label{def:join}
A representation $\theta$ of a poset $\algebra{P}$ over the base $X$ is \defn{join complete} if, for every subset $S$ of $\algebra{P}$, if $\join S$ exists, then
\[\theta(\join S) = \bigcup \theta[S].\]
\end{definition}

Note that $S$ is required to be nonempty in \Cref{def:meet}, but not in \Cref{def:join}. For representations of Boolean algebras as fields of sets, the notions of meet complete and join complete are equivalent, so in this case we may simply use the adjective \defn{complete}.

Note that if $\algebra{A}$ is an algebra of the signature $(\compo, \bmeet, \A)$ and $\algebra{A}$ is representable by partial functions, then $\algebra{A}$ must have a least element, $0$, given by $\A(a) \compo a$ for any $a \in \algebra{A}$ and any representation must represent $0$ with the empty set. Similarly $\D \coloneqq \A^2$ must be represented by the set-theoretic domain operation.

The following lemma demonstrates the utility of the particular signature $(\compo, \bmeet, \A)$. The similarity of representable $(\compo, \bmeet, \A)$-algebras to Boolean algebras allows results from the theory of Boolean algebras to be imported into the setting of $(\compo, \bmeet, \A)$-algebras.

\begin{lemma}\label{lemma:boolean}
Let $\algebra{A}$ be an algebra of the signature $(\compo, \bmeet, \A)$. If $\algebra{A}$ is representable by partial functions, then for every $a \in \algebra{A}$, the set $\down a$, with least element $0$, greatest element $a$, meet given by $\bmeet$ and complementation given by $\compl{b} \coloneqq \A(b) \compo a$ is a Boolean algebra. Any representation $\theta$ of $\algebra{A}$ by partial functions restricts to a representation of $\down a$ as a field of sets over $\theta(a)$. If $\theta$ is a meet-complete or join-complete representation, then the representation of $\down a$ is complete.
\end{lemma}

\begin{proof}
If $\theta$ is a representation of $\algebra{A}$ by partial functions, then $b \leq a \implies \theta(b) \subseteq \theta(a)$, so $\theta$ does indeed map elements of $\down a$ to subsets of $\theta(a)$. We have $b, c \in \down a \implies b \bmeet c \in \down a$ and $\theta(b \bmeet c) = \theta(b) \cap \theta(c)$ is always true by the definition of functional representability. For $b \leq a$
\[\theta(\compl{b}) = \theta(\A(b) \compo a) = \A(\theta(b)) \compo \theta(a) = \theta(a) \setminus \theta(b),\]
so $\compl{b} \in \down a$ and $\theta(\compl{b}) = \stcomp{\theta(b)}$, where the set complement is taken relative to $\theta(a)$. Hence the restriction of $\theta$ to $\down a$ is a representation of $(\down a, 0, a, \bmeet, \compl{\phantom{c}})$ as a field of sets over $\theta(a)$ (from which it follows that $\down a$ is a Boolean algebra).

Suppose $\theta$ is meet complete. If $S$ is a nonempty subset of $\down a$, then all lower bounds for $S$ in $\algebra{A}$ are also in $\down a$. Hence if $\meet_{\down a} S$ exists then it equals $\meet_{\algebra{A}} S$, and so $\theta(\meet_{\down a} S) = \bigcap \theta[S]$. So the representation of $\down a$ is complete.

Suppose that $\theta$ is join complete, $S \subseteq \down a$ and $\join_{\down a} S$ exists. If $c \in \algebra{A}$ and $c$ is an upper bound for $S$, then $c \geq c \bmeet a \geq \join_{\down a} S$. Hence $\join_{\down a} S = \join_{\algebra{A}} S$, giving $\theta(\join_{\down a} S) = \theta(\join_{\algebra{A}} S) = \bigcup \theta[S]$. So the representation of $\down a$ is complete.
\end{proof}

\begin{corollary}\label{cor:1}
Let $\algebra{A}$ be an algebra of the signature $(\compo, \bmeet, \A)$ and $\theta$ be a representation of $\algebra{A}$ by partial functions. If $\theta$ is meet complete, then it is join complete.
\end{corollary}

\begin{proof}
Suppose that $\theta$ is meet complete. Let $S$ be a subset of $\algebra{A}$ and suppose that $\join_{\algebra{A}} S$ exists. Let $a = \join_{\algebra{A}} S$. Then
\[\theta(\join_{\algebra{A}} S) = \theta(\join_{\down a} S) = \bigcup \theta[S]\text{.}\qedhere\]
\end{proof}

\begin{corollary}\label{cor:2}
Let $\algebra{A}$ be an algebra of the signature $(\compo, \bmeet, \A)$ and $\theta$ be a representation of $\algebra{A}$ by partial functions. If $\theta$ is join complete, then it is meet complete.
\end{corollary}

\begin{proof}
Suppose that $\theta$ is join complete. Let $S$ be a nonempty subset of $\algebra{A}$ and suppose that $\meet_{\algebra{A}} S$ exists. As $S$ is nonempty, we can find $s \in S$. Then
\[\theta(\meet_{\algebra{A}} S) = \theta(\meet_{\algebra{A}} (S \bmeet \{s\})) = \theta(\meet_{\down{s}} (S \bmeet \{s\})) = \bigcap \theta[S \bmeet \{s\}] = \bigcap \theta[S]\text{.}\qedhere\]
\end{proof}

Corollaries \ref{cor:1} and \ref{cor:2} tell us that, just as for representations of Boolean algebras, we can describe representations of $(\compo, \bmeet, \A)$-algebras by partial functions as \defn{complete}, without any risk of confusion about whether we mean meet complete or join complete.

\section{Atomicity}\label{atomicity}

We begin our investigation of the complete representation class by considering the property of being atomic, both for algebras and for representations.

\begin{definition}
Let $\algebra{P}$ be a poset with a least element, $0$. An \defn{atom} of $\algebra{P}$ is a minimal nonzero element of $\algebra{P}$. We say that $\algebra{P}$ is \defn{atomic} if every nonzero element is greater than or equal to an atom.
\end{definition}

If $\algebra{P}$ is a poset, then $\At(\algebra{P})$ denotes the set of atoms of $\algebra{P}$.

We noted in the proof of \Cref{lemma:boolean} that representations of $(\compo, \bmeet, \A)$-algebras necessarily represent the partial order by set inclusion. The following definition is meaningful for any notion of representation where this is the case.

\begin{definition}
Let $\algebra{P}$ be a poset with a least element and let $\theta$ be a representation of $\algebra{P}$. Then $\theta$ is \defn{atomic} if $x \in \theta(a)$ for some $a \in \algebra{P}$ implies $x \in \theta(b)$ for some atom $b$ of $\algebra{P}$.
\end{definition}

We will need the following theorem.

\begin{theorem}[Hirsch and Hodkinson \cite{journals/jsyml/HirschH97a}]\label{thm:hirsch-hodkinson} Let $\algebra{B}$ be a Boolean algebra. A representation of $\algebra{B}$ as a field of sets is atomic if and only if it is complete.
\end{theorem}

\begin{proposition}\label{prop:at-com}
Let $\algebra{A}$ be an algebra of the signature $(\compo, \bmeet, \A)$ and $\theta$ be a representation of $\algebra{A}$ by partial functions. Then $\theta$ is atomic if and only if it is complete.
\end{proposition}

\begin{proof}
Suppose that $\theta$ is atomic, $S$ is a nonempty subset of $\algebra{A}$ and $\meet S$ exists. It is always true that $\theta(\meet S) \subseteq \bigcap \theta[S]$, regardless of whether or not $\theta$ is atomic. For the reverse inclusion, we have
\[\begin{array}{cll}
 & (x, y) \in \bigcap \theta[S]
\\ \implies & (x, y) \in \theta(s) & \text{for all }s \in S
\\ \implies & (x, y) \in \theta(a) & \text{for some atom }a\text{ such that }(\forall s \in S)\text{ } a \leq s
\\ \implies & (x, y) \in \theta(a) & \text{for some atom }a\text{ such that } a \leq \meet S
\\ \implies &  (x, y) \in \theta(\meet S)\text{.}
\end{array}\]
The third line follows from the second because, taking an $s_0 \in S$ and an atom $a$ below $s_0$ with $(x, y) \in \theta(a)$, we have $(x, y) \in \theta(a \bmeet s)$ for any $s \in S$. So for all $s \in S$, the element $a \bmeet s$ is nonzero, so equals $a$, by atomicity of $a$, giving $a \leq s$.

Conversely, suppose that $\theta$ is complete. Let $(x, y)$ be a pair contained in $\theta(a)$ for some $a \in \algebra{A}$. By \Cref{lemma:boolean}, the map $\theta$ restricts to a complete representation of $\down a$ as a field of sets. Hence, by \Cref{thm:hirsch-hodkinson}, $(x, y) \in \theta(b)$ for some atom $b$ of the Boolean algebra $\down a$. Since an atom of $\down a$ is clearly an atom of $\algebra{A}$, the representation $\theta$ is atomic.
\end{proof}

\begin{corollary}\label{cor:atomic}
Let $\algebra{A}$ be an algebra of the signature $(\compo, \bmeet, \A)$. If $\algebra{A}$ is completely representable by partial functions then $\algebra{A}$ is atomic.
\end{corollary}

\begin{proof}
Let $a$ be a nonzero element of $\algebra{A}$. Let $\theta$ be any complete representation of $\algebra{A}$. Then $\emptyset = \theta(0) \neq \theta(a)$, so there exists $(x, y) \in \theta(a)$. By \Cref{prop:at-com}, the map $\theta$ is atomic, so $(x, y) \in \theta(b)$ for some atom $b$ in $\algebra{A}$. Then $(x, y) \in \theta(a \bmeet b)$, so $a \bmeet b > 0$, from which we may conclude that the atom $b$ satisfies $b \leq a$.
\end{proof}

So far we have exploited the Boolean algebras that are contained in any representable $(\compo, \bmeet, \A)$-algebra. But we can also travel in the opposite direction and interpret any Boolean algebra as an algebra of the signature $(\compo, \bmeet, \A)$, by using the Boolean meet for both the composition and meet operations and Boolean complement for antidomain. Again this enables us to easily prove results about $(\compo, \bmeet, \A)$-algebras using results about Boolean algebras.

We know by the following argument that a Boolean algebra, $\algebra{B}$, viewed as an algebra of the signature $(\compo, \bmeet, \A)$, is representable by partial functions. By Stone's representation theorem we may assume that $\algebra{B}$ is a field of sets. Then the set of all identity functions on elements of $\algebra{B}$ forms a representation of $\algebra{B}$ by partial functions. Using the same argument, it is easy to see that a Boolean algebra is completely representable as a field of sets if and only if it is completely representable by partial functions.

Hirsch and Hodkinson used \Cref{thm:hirsch-hodkinson} to identify those Boolean algebras completely representable as fields of sets as precisely the atomic Boolean algebras.\footnote{This result had also been discovered previously by Abian \cite{Abian01051971}.} Hence a Boolean algebra is completely representable by partial functions if and only if it is atomic. The following proposition uses this fact to prove various negative results about the axiomatisability of the class of completely representable $(\compo, \bmeet, \A)$-algebras.

\begin{proposition}\label{closure}
The class of $(\compo, \bmeet, \A)$-algebras that are completely representable by partial functions is not closed with respect to the operations shown in the following table and so is not axiomatisable by first-order theories of the indicated corresponding form.

\hfill \\
\begin{tabular}{r  l  l }

& Operation & Axiomatisation \\[1mm]

(i) & subalgebra & universal \\

(ii) & directed union & universal-existential \\

(iii) & homomorphism & positive \\

\end{tabular}
\end{proposition}

\begin{proof}
In each case we use the fact, which we noted previously, that a Boolean algebra is completely representable by partial functions if and only if it is atomic.
\begin{enumerate}[(i)]
\item
We show that the class is not closed under subalgebras. It follows that the class cannot be axiomatised by any universal first-order theory. Let $\algebra{B}$ be any non-atomic Boolean algebra, for example the countable atomless Boolean algebra, which is unique up to isomorphism. By Stone's representation theorem we may assume that $\algebra{B}$ is a field of sets, with base $X$ say. Then $\algebra{B}$ is a subalgebra of $\powerset(X)$ and $\powerset(X)$ is atomic, but $\algebra{B}$ is not.
\item
We show that the class is not closed under directed unions. It follows that the class cannot be axiomatised by any universal-existential first-order theory. Again, let $\algebra{B}$ be any non-atomic Boolean algebra. Then $\algebra{B}$ is the union of its finitely generated subalgebras, which form a directed set of algebras. The finitely generated subalgebras, being Boolean algebras, are finite and hence atomic. So we have, as required, a directed set of atomic Boolean algebras whose union is not atomic.
\item
We show that the class is not closed under homomorphic images. It follows that the class cannot be axiomatised by any positive first-order theory. Let $X$ be any infinite set and $I$ the ideal of $\powerset(X)$ consisting of finite subsets of $X$. Then $\powerset(X)$ is atomic, but the quotient $\powerset(X) / I$ is atomless and nontrivial and so is not atomic.
\qedhere
\end{enumerate}
\end{proof}

Since we have mentioned the subalgebra and homomorphism operations, we note that the class of completely representable $(\compo, \bmeet, \A)$-algebras \emph{is} closed under direct products. Indeed, it is routine to verify that given complete representations of each factor in a product we can form a complete representation of the product using disjoint unions in the obvious way.

\begin{proposition}\label{axioms}
The class of $(\compo, \bmeet, \A)$-algebras that are completely representable by partial functions is not axiomatisable by any existential-universal-existential first-order theory. 
\end{proposition}

\begin{proof}
Let $\algebra{B}$ be any atomic Boolean algebra with an infinite number of atoms and $\algebra{B'}$ be any Boolean algebra that is not atomic, but that also has an infinite number of atoms. We will show that $\algebra{B'}$ satisfies any existential-universal-existential sentence satisfied by $\algebra{B}$. Since $\algebra{B}$ is completely representable by partial functions and $\algebra{B'}$ is not, this shows that the complete representation class cannot be axiomatised by any existential-universal-existential theory.

We will show that for certain Ehrenfeucht-Fra\"{i}ss\'{e} games, duplicator has a winning strategy. For an overview of Ehrenfeucht-Fra\"{i}ss\'{e} games see, for example, \cite{HodgesW:modt}. Briefly, two players, spoiler and duplicator, take turns to choose elements from two algebras. Duplicator wins if the two sequences of choices determine an isomorphism between the subalgebras generated by all the elements chosen.

Consider the game in which spoiler must in the first round choose $n_1$ elements of $\algebra{B}$, in the second round $n_2$ elements of $\algebra{B'}$ and in the third and final round $n_3$ elements of $\algebra{B}$. Each round, duplicator responds with corresponding choices from the other algebra. Let $\varphi$ be any sentence in prenex normal form whose quantifiers are, starting from the outermost, $n_1$ universals, then $n_2$ existentials and finally $n_3$ universals. It is not hard to convince oneself that if duplicator has a winning strategy for the game then $\algebra{B'} \models \varphi \implies \algebra{B} \models \varphi$. Hence if duplicator has a winning strategy for all games of this form---where spoiler chooses finite numbers of elements from $\algebra{B}$ then $\algebra{B'}$ then $\algebra{B}$---then all universal-existential-universal sentences satisfied by $\algebra{B'}$ are satisfied by $\algebra{B}$. Equivalently, $\algebra{B'}$ satisfies any existential-universal-existential sentence satisfied by $\algebra{B}$, which is what we are aiming to show.

Since our algebras are Boolean algebras, a choice of a finite number of elements from one of the algebras generates a finite subalgebra, with a finite number of atoms. The atoms form a partition, that is, a sequence $(a_1,\ldots,a_n)$ of nonzero elements with $\join_i a_i = 1$ and $a_i \bmeet a_j = 0$ for all $i \neq j$. As the game progresses and more elements are chosen, the partition is refined---the elements of the partition are (finitely) further subdivided. The elements the two players have actually chosen are all uniquely expressible as a join of some subset of the partition.

Suppose that, throughout the game, duplicator is able to maintain a correspondence between the partitions on the two algebras. That is, if spoiler subdivides an element $a$ of the existing partition into $(a_1,\ldots,a_n)$ then the element corresponding to $a$ should be partitioned into a corresponding $(a_1',\ldots,a_n')$. Then clearly this determines a winning sequence of moves for duplicator: each of spoiler's choices is the join of some subset of one partition and duplicator's choice should be the join of the corresponding elements of the other partition. At the end of the game there will exist an isomorphism between the generated subalgebras that sends each element chosen during the game to the corresponding choice from the other algebra. Hence a strategy for maintaining a correspondence between the two partitions provides a winning strategy for duplicator.

For an element $a$ of $\algebra{B}$ or $\algebra{B'}$ we will say that $a$ is of size $n$, for finite $n$, if $a$ is the join of $n$ distinct atoms, otherwise $a$ is of infinite size. Duplicator can maintain a correspondence by playing as follows.

\begin{description}
\item[Round 1] (Spoiler plays on atomic algebra, duplicator on non-atomic) Duplicator should simply provide a partition with matching sizes.

\item[Round 2] (Spoiler non-atomic, duplicator atomic) For subdivisions of elements of finite size, duplicator can provide a subdivision with matching sizes. For subdivisions of elements of infinite size, there is necessarily at least one element in the subdivision of infinite size---duplicator should select one such, match everything else with distinct single atoms and match this infinite size element with what remains on the atomic side.

\item[Round 3] (Spoiler atomic, duplicator non-atomic) At the start of this round every element of the partition of the atomic algebra is matched with something of greater or equal size on the non-atomic side. Hence duplicator can easily provide matching subdivisions.\qedhere
\end{description}
\end{proof}

\section{Distributivity}\label{dist}

We now turn our attention to the validity of various distributive laws with respect to the classes of representable and completely representable $(\compo, \bmeet, \A)$-algebras. We give the first definition that we will use. Other distributive properties that we refer to later are defined similarly. For distributive properties `over meets' it should be assumed that definitions only require that the relevant equation holds when \emph{nonempty} subsets are used.

\begin{definition}
Let $\algebra{P}$ be a poset and $*$ be a binary operation on $\algebra{P}$. We say that $*$ is \defn{completely right-distributive over joins} if, for any subset $S$ of $\algebra{P}$ and any $a \in \algebra{P}$, if $\join S$ exists, then
\[\join S * a = \join(S * \{a\})\text{.}\]
\end{definition}

\begin{proposition}
Let $\algebra{A}$ be an algebra of the signature $(\compo, \bmeet, \A)$ that is representable by partial functions. Then composition is completely right-distributive over joins.
\end{proposition}

\begin{proof}
As $\algebra{A}$ is representable, we may assume the elements of $\algebra{A}$ are partial functions. Let $S$ be a subset of $\algebra{A}$ such that $\join S$ exists and let $a \in \algebra{A}$.

Firstly, for all $s \in S$ we have $\join S \compo a \geq s \compo a$  and so $\join S \compo a$ is an upper bound for $S \compo \{a\}$.

Now suppose that for all $s \in S$, the element $b \in \algebra{A}$ satisfies $b \geq s \compo a$. For $s \in S$, suppose $s$ is defined on $x$ and let $s(x) = y$. If $a$ is defined on $y$, then $s \compo a$ is defined on $x$, so, since $b \geq s \compo a$ and $\join S \compo a \geq s \compo a$, in this case $b \bmeet (\join S \compo a)$ is defined on $x$. If $a$ is \emph{not} defined on $y$ then, as $(\join S)(x) = y$, in this case $\join S \compo a$ is not defined on $x$. Hence the sub-identity function $\D(b \bmeet (\join S \compo a)) \bjoin \A(\join S \compo a)$ is defined on the entire domain of $s$. Therefore
\[(\D(b \bmeet (\join S \compo a)) \bjoin \A(\join S \compo a)) \compo \join S \geq s\text{.}\]
Since $s$ was an arbitrary element of $S$, we have
\[(\D(b \bmeet (\join S \compo a)) \bjoin \A(\join S \compo a)) \compo \join S \geq \join S\]
and so
\[(\D(b \bmeet (\join S \compo a)) \bjoin \A(\join S \compo a)) \compo \join S = \join S \text{.}\]
Therefore
\begin{align*}
\D(b \bmeet (\join S \compo a)) \compo \join S \compo a &= (\D(b \bmeet (\join S \compo a)) \bjoin \A(\join S \compo a)) \compo \join S \compo a
\\ &= \join S \compo a\text{,}
\end{align*}
which says that wherever the function $\join S \compo a$ is defined, it agrees with the function $b$, that is to say $b \geq \join S \compo a$. So $\join S \compo a$ is the least upper bound for $\join(S \compo \{a\})$.
\end{proof}

\begin{remark}\label{dist-laws}
For $(\compo, \bmeet, \A)$-algebras representable by partial functions it is easy to see that the following two laws hold.
\begin{enumerate}[(i)]
\item\label{law:ljoin}
For finite $S$, if $\join S$ exists, then
\begin{align*}a \compo \join S = \join(\{a\} \compo S)\text{.} && \text{(composition is left-distributive over joins)}\end{align*}
\item\label{law:lmeet}
For finite, nonempty $S$,
\begin{align*}a \compo \meet S = \meet(\{a\} \compo S)\text{.} && \text{(composition is left-distributive over meets)}\end{align*}
\end{enumerate}
\end{remark}

We now give an example that shows that the these distributive laws cannot, in general, be extended to arbitrary joins and meets. We will use this example to show that there exist $(\compo, \bmeet, \A)$-algebras that are representable as partial functions, and atomic, but have no atomic representation.

\begin{example}\label{eg:noncdist}
Consider the following concrete algebra of partial functions, $\algebra{F}$. Its domain is the disjoint union of a one element set, $\{p\}$, and $\mathbb{N}_\infty \coloneqq \mathbb{N} \cup \{\infty\}$. Let $\mathcal{S}$ be all the subsets of $\mathbb{N}_\infty$ that are either finite and do not contain $\infty$, or cofinite and contain $\infty$. The elements of $\algebra{F}$ are precisely the following functions.
\begin{enumerate}
\item Restrictions of the identity to $A \cup B$ where $A \subseteq \{p\}$ and $B \in \mathcal{S}$.

\item The function $f$, defined only on $p$ and taking $p$ to $\infty$.
\end{enumerate}

One can check that $\algebra{F}$ is closed under the operations of intersection, composition and antidomain, that $\algebra{F}$ is atomic and that $f$ is an atom.

For $i \in \mathbb{N}$, let $g_i$ be the restriction of the identity to $\{1,\ldots,i\}$. Then $\join_i g_i$ exists and is equal to the identity restricted to $\mathbb{N}_\infty$. So
\[f \compo \join_{i \in \mathbb{N}} g_i = f \neq \emptyset = \join_{i \in \mathbb{N}} (f \compo g_i)\text{.}\]

For $i \in \mathbb{N}$, let $h_i$ be the restriction of the identity to $\{i,\ldots\} \cup \{\infty\}$. Then $\meet_i h_i$ exists and is equal to the nowhere-defined function. So
\[f \compo \meet_{i \in \mathbb{N}} h_i = \emptyset \neq f = \meet_{i \in \mathbb{N}} (f \compo h_i)\text{.}\]
\end{example}

\begin{lemma}\label{lemma:compdist}
Let $\algebra{A}$ be an algebra of the signature $(\compo, \bmeet, \A)$ that is completely representable by partial functions. Then composition in $\algebra{A}$ is completely left-distributive over joins and completely left-distributive over meets.
\end{lemma}

\begin{proof}
First we prove that composition is completely left-distributive over joins. Let $S$ be a subset of $\algebra{A}$ such that $\join S$ exists and let $a \in \algebra{A}$. Let $\theta$ be any complete representation of $\algebra{A}$. Suppose that for all $s \in S$ the element $b \in \algebra{A}$ satisfies $b \geq a \compo s$. Then for all $s \in S$ we have $\theta(b) \supseteq \theta(a \compo s)$. Hence
\begin{align*}
\theta(b) &\supseteq \bigcup \theta[\{a\} \compo S] \\
&= \bigcup (\{\theta(a)\} \compo \theta[S]) \\
&= \theta(a) \compo \bigcup \theta[S] \\
&= \theta(a \compo \join S)\text{.} 
\end{align*}
The second equality is a true property of any collection of functions, indeed of any collection of relations. We conclude that $b \geq a \compo \join S$ and hence $a \compo \join S$ is the least upper bound for $\{a\} \compo S$.

The proof that composition is completely left-distributive over meets is similar. Let $S$ be a nonempty subset of $\algebra{A}$ such that $\meet S$ exists and let $a \in \algebra{A}$. Let $\theta$ be any complete representation of $\algebra{A}$. Suppose that for all $s \in S$, the element $b \in \algebra{A}$ satisfies $b \leq a \compo s$. Then for all $s \in S$, we have $\theta(b) \subseteq \theta(a \compo s)$. Hence
\begin{align*}
\theta(b) &\subseteq \bigcap \theta[\{a\} \compo S] \\
&= \bigcap (\{\theta(a)\} \compo \theta[S]) \\
&= \theta(a) \compo \bigcap \theta[S] \\
&= \theta(a \compo \meet S)\text{.} 
\end{align*}
This time the second equality holds only because we are working with functions. It is not, in general, a true property of relations. We conclude from the above that $b \geq a \compo \meet S$ and hence $a \compo \meet S$ is the greatest lower bound for $\{a\} \compo S$.
\end{proof}

\begin{proposition}
There exist $(\compo, \bmeet, \A)$-algebras that are representable by partial functions, and atomic, but have no atomic representation.
\end{proposition}

\begin{proof}
Let $\algebra{F}$ be the algebra of \Cref{eg:noncdist}. Since $\algebra{F}$ \emph{is} an algebra of partial functions, it is certainly representable by partial functions. We have already mentioned that $\algebra{F}$ is atomic. We have demonstrated that composition in $\algebra{F}$ is neither completely left-distributive over joins nor over meets. Hence, by \Cref{lemma:compdist}, $\algebra{F}$ has no complete representation. So, by \Cref{prop:at-com}, $\algebra{F}$ has no atomic representation.
\end{proof}

To make the discussion of distributive laws comprehensive we finish by mentioning right-distributivity of composition over meets. Here the weakest possible result, that the finite version of the law is valid for completely representable algebras, does not hold for representation by partial functions. In the algebra of partial functions shown in \Cref{picture}, where sub-identity elements are omitted, we have
\[(f_1 \bmeet f_2) \compo g = 0 \compo g = 0 \neq h = h \bmeet h = (f_1 \compo g) \bmeet (f_2 \compo g)\text{.}\]
The algebra is completely representable because it is already an algebra of partial functions and it is finite.

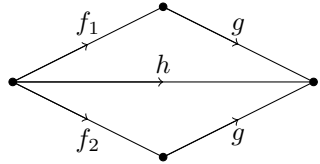
\begin{figure}[H]
\centering
\begin{tikzpicture}
\begin{scope}[->]
\draw (0,0)--(1,0.5);
\draw (0,0)--(1,-0.5);
\draw (0,0)--(2,0);
\draw (2,1)--(3,0.5);
\draw (2,-1)--(3,-0.5);
\end{scope}
\draw (0,0)--(2,1);
\draw (0,0)--(2,-1);
\draw (0,0)--(4,0);
\draw (2,1)--(4,0);
\draw (2,-1)--(4,0);
\draw[fill] (0,0) circle [radius=0.05];
\draw[fill] (4,0) circle [radius=0.05];
\draw[fill] (2,1) circle [radius=0.05];
\draw[fill] (2,-1) circle [radius=0.05];
\node [above] at (1,0.5) {$f_1$};
\node [above] at (3,0.5) {$g$};
\node [below] at (1,-0.5) {$f_2$};
\node [below] at (3,-0.5) {$g$};
\node [above] at (2,0) {$h$};
\end{tikzpicture}
\caption{An algebra refuting right-distributivity over meets}\label{picture}
\end{figure}

\section{A Representation}\label{arep}

We have seen that for an algebra of the signature $(\compo, \bmeet, \A)$ to be completely representable by partial functions it is necessary for it to be representable by partial functions and atomic and for composition to by completely left-distributive over joins. Next we show that these conditions are also sufficient.

\begin{proposition}\label{prop:rep}
Let $\algebra{A}$ be an algebra of the signature $(\compo, \bmeet, \A)$. Suppose $\algebra{A}$ is representable by partial functions and atomic and that composition is completely left-distributive over joins. For each $a \in \algebra{A}$, let $\theta(a)$ be the following partial function on $\At(\algebra{A})$.
\[\theta(a)(x) =
\begin{cases}
x \compo a & \mathrm{if }\text{ }x \compo a \neq 0 \\
\mathrm{undefined} & \mathrm{otherwise}
\end{cases}\]
Then $\theta$ is a complete representation of $\algebra{A}$ by partial functions, with base $\At(\algebra{A})$.
\end{proposition}

\begin{proof}
We first need to show that, for each $a \in \algebra{A}$, the partial function $\theta(a)$ maps into $\At(\algebra{A})$. Let $x$ be an atom and suppose that $x \compo a$ is nonzero. Let $b \in \algebra{A}$ and suppose $b \leq x \compo a$. Then $\D(b) \leq \D(x \compo a) \leq \D(x)$. Hence if $\D(b) \compo x = 0$ then $b = 0$. If $\D(b) \compo x > 0$, then we must have $\D(b) \compo x = x$ and hence $\D(b) = \D(x \compo a) = \D(x)$. Therefore $b = x \compo a$. So $x \compo a$ is an atom.

To show that $\theta$ represents composition correctly, let $a, b \in \algebra{A}$ and $x \in \At(\algebra{A})$. Then clearly $\theta(a \compo b)(x) = \theta(a) \compo \theta(b)(x)$ if both sides are defined. The left-hand side is defined precisely when $x \compo a \compo b$ is nonzero and the right-hand side when $x \compo a$ and $x \compo a \compo b$ are both nonzero. Since $x \compo a \compo b \neq 0$ implies $x \compo a \neq 0$, the domains of definition are the same.

To show that $\theta$ represents binary meet correctly, let $a, b \in \algebra{A}$ and $x, y \in \At(\algebra{A})$. Then
\begin{align*}
&(x, y) \in \theta(a \bmeet b) \\
\implies \quad & (x, y) \in \theta(a)\text{ and }(x, y) \in \theta(b) & \text{as }a, b \geq a \bmeet b \\
\implies \quad & (x, y) \in \theta(a) \cap \theta(b) 
\shortintertext{and}
& (x, y) \in \theta(a) \cap \theta(b) \\
\implies \quad & x \compo a = y\text{ and }x \compo b = y \\
\implies \quad & (x \compo a) \bmeet (x \compo b) = y \\
\implies \quad & x \compo (a \bmeet b) = y &\text{ by \Cref{dist-laws}}\\
\implies \quad & (x, y) \in \theta(a \bmeet b)\text{.}
\end{align*}

To show that antidomain is represented correctly, let $a \in \algebra{A}$ and $x \in \At(\algebra{A})$. Then $0 < \theta(\A(a))(x) = x \compo \A(a) \leq x$ if $\theta(\A(a))(x)$ is defined. Since $x$ is an atom we have, in this case, $\theta(\A(a))(x) = x$. The partial function $\A(\theta(a))$ is also a restriction of the identity function. The domains of $\theta(\A(a))$ and $\A(\theta(a))$ are the same, since we have seen that $\theta(\A(a))(x)$ is defined precisely when $x \compo \A(a) = x$, which is when $x \compo a = 0$, which is precisely when $\A(\theta(a))(x)$ is defined. This completes the proof that $\theta$ is a representation of $\algebra{A}$ by partial functions.

Finally, we show that the representation $\theta$ is complete. Let $S$ be a subset of $\algebra{A}$ such that $\join S$ exists. Let $x, y \in \At(\algebra{A})$. Then
\begin{align*}
 & (x, y) \in \bigcup \theta[S] \\
\implies \quad & (x, y) \in \theta(s) &\quad& \text{for some }s \in S \\
\implies \quad & (x, y) \in \theta(\join S) &\quad& \text{as }\join S \geq s
\shortintertext{and}
 & (x, y) \in \theta(\join S) \\
\implies \quad & x \compo \join S = y & \\
\implies \quad & \join(\{x\} \compo S) = y &\quad& \text{as }\compo\text{ is completely left-distributive over joins} \\
\implies \quad & x \compo s = y &\quad& \text{for some }s \in S\text{, since }y\text{ is an atom} \\
\implies \quad & (x, y) \in \theta(s) &\quad& \text{for some }s \in S \\
\implies \quad &  (x, y) \in \bigcup \theta[S]\text{.} 
\end{align*}
Hence $\theta(\join S) = \bigcup \theta[S]$.
\end{proof}

\section{Axiomatising the Class}\label{axiom}

In this final section, we use the conditions for complete representability that we have uncovered to obtain a finite first-order axiomatisation of the complete-representation class.

\begin{definition}
A poset $\algebra{P}$ is \defn{atomistic} if its atoms are join dense in $\algebra{P}$. That is to say that every element of $\algebra{P}$ is the join of the atoms less than or equal to it.
\end{definition}

Clearly any atomistic poset is atomic. For $(\compo, \bmeet, \A)$-algebras representable by partial functions, the converse is also true.

\begin{lemma}\label{atomistic}
Let $\algebra{A}$ be an algebra of the signature $(\compo, \bmeet, \A)$ that is representable by partial functions. If $\algebra{A}$ is atomic, then it is atomistic.
\end{lemma}

\begin{proof}
Suppose $\algebra{A}$ is atomic and let $a \in \algebra{A}$. By \Cref{lemma:boolean}, the algebra $\down a$ is a Boolean algebra and clearly it is atomic. It is well-known that atomic Boolean algebras are atomistic. So we have
\[
a = \join_{\down a} \{ x \in \At(\down a) \mid x \leq a \} = \join_{\algebra{A}} \{ x \in \At(\down a) \mid x \leq a \} = \join_{\algebra{A}} \{ x \in \At(\algebra{A}) \mid x \leq a \}\text{.}
\]
The second equality holds because any upper bound $c \in \algebra{A}$ for $\{ x \in \At(\down a) \mid x \leq a \}$ is above an upper bound in $\down a$, for example $c \bmeet a$. Hence the least upper bound in $\down a$ is least in $\algebra{A}$ also.
\end{proof}

\begin{lemma}\label{lemma:phi}
Let $\algebra{A}$ be an algebra of the signature $(\compo, \bmeet, \A)$ that is representable by partial functions and atomic. Let $\varphi$ be the first-order sentence asserting that for any $a, b, c$, if $c \geq a \compo x$ for all atoms $x$ less than or equal to $b$, then $c \geq a \compo b$. Then composition is completely left-distributive over joins if and only if $\algebra{A} \models \varphi$.
\end{lemma}

\begin{proof}
Suppose first that composition is completely left-distributive over joins. As $\algebra{A}$ is atomic it is atomistic. So for any $a, b \in \algebra{A}$ we have
\[
a \compo b = a \compo \join \{x \in \At(\algebra{A}) \mid x \leq b\} = \join (\{a\} \compo \{x \in \At(\algebra{A}) \mid x \leq b\})
\]
and so $\varphi$ holds.

Now suppose that $\algebra{A} \models \varphi$. Let $a \in \algebra{A}$ and let $S$ be a subset of $\algebra{A}$ such that $\join S$ exists. Then certainly $a \compo \join S$ is an upper bound for $\{a\} \compo S$. To show it is the least upper bound, let $c$ be an arbitrary upper bound for $\{a\} \compo S$. Then
\begin{align*}
&\text{for all }s \in S &\quad& c \geq a \compo s \\
\implies \quad& \text{for all }s \in S\text{ and }x \in \At(\down \join S)\text{ with }x \leq s&& c \geq a \compo x  \\
\implies \quad& \text{for all }x \in \At(\down \join S) && c \geq a \compo x \\
\implies \quad& \text{for all }x \in \At(\algebra{A})\text{ with }x \leq \join S && c \geq a \compo x \\
\implies \quad&&& c \geq a \compo \join S\text{.}
\end{align*}
The third line follows from the second because $x \in \At(\down \join S)$ implies $x \leq s$ for some $s \in S$. To see this, consider the Boolean algebra $\down \join S$. When $x$ is an atom, $x \nleq s$ if and only if $x \bmeet s = 0$, which is equivalent to $\compl{x} \geq s$. So if $x \nleq s$ for all $s \in S$ then $\compl{x} \geq \join S$, forcing $x$ to be zero---a contradiction. The fifth line can be seen to follow from the fourth by first writing $\join S$ as the join of the atoms below it and then using complete left-distributivity.
\end{proof}

We now have everything we need to prove our main result.

\begin{theorem}
The class of $(\compo, \bmeet, \A)$-algebras that are completely representable by partial functions is a basic elementary class.
\end{theorem}

\begin{proof}
By \Cref{cor:atomic}, \Cref{lemma:compdist} and \Cref{prop:rep}, an algebra of the signature $(\compo, \bmeet, \A)$ is completely representable by partial functions if and only if it is representable by partial functions, atomic and composition is completely left-distributive over joins. By \Cref{thm:jackson-stokes}, the property of being representable by partial functions is characterised by a finite set of first-order sentences. The property of being atomic is easily written as a first-order sentence. By \Cref{lemma:phi}, in the presence of the axioms for the first two properties, the property that composition is completely left-distributive over joins can be written as a first-order sentence.
\end{proof}

Any attempt at writing down our axioms will readily reveal that each can be expressed in a universal-existential-universal form. We know from \Cref{axioms} that no existential-universal-existential axiomatisation is possible, hence we have determined the precise amount of quantifier alternation necessary to axiomatise the class.

Note that if range had been included in our signature then the function $\theta$ in \Cref{prop:rep} would not be a representation, as it would not represent range correctly. \Cref{fig:range} shows how this can happen. The atom $f$ satisfies $f \compo \R(g) = f$ and so $(f, f) \in \theta(\R(g))$, but there is no $h$ such that $h \compo g = f$ and so $(f, f) \not\in \R(\theta(g))$. Hence questions about the axiomatisability of the complete representation class for the signature $(\compo,  \bmeet, \A, \R)$ remain open. Equally for the less expressive signature $(\compo, \bmeet, \D)$, where the meet-complete and join-complete representations do not coincide.

\begin{figure}[H]
\centering
\begin{tikzpicture}
\draw (0,1)--(2,0);
\draw[->](0,1)--(1,0.5);
\draw[fill] (0,1) circle [radius=0.05];
\node [above] at (1,0.5) {$f$};

\draw (0,-1)--(2,0);
\draw[->](0,-1)--(1,-0.5);
\draw[fill] (0,-1) circle [radius=0.05];
\draw[fill] (2,0) circle [radius=0.05];
\node [below] at (1,-0.5) {$g$};
\end{tikzpicture}
\caption{Algebra for which $\theta$ does not represent range correctly}\label{fig:range}
\end{figure}
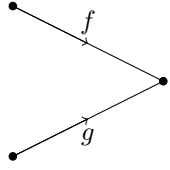

\bibliographystyle{amsplain}

\bibliography{Complete_representation_by_partial_functions}

\end{document}